\documentclass[12pt]{article}

\usepackage{amsmath, amssymb, amsthm, amscd, geometry, amsfonts, graphicx, setspace,tikz, pifont, enumerate, cancel,  mathrsfs, latexsym, url ,setspace, mathtools}

\usepackage[normalem]{ulem}
\usepackage{xcolor}
\newcommand\redsout{\bgroup\markoverwith{\textcolor{red}{\rule[0.5ex]{2pt}{0.4pt}}}\ULon}

\newtheorem{theorem}{Theorem}
\newtheorem*{theorem*}{Theorem}
\newtheorem{proposition}{Proposition}

\newtheorem{corollary}{Corollary}
\theoremstyle{remark}
\newtheorem{remark}{Remark}
\theoremstyle{definition}
\newtheorem{definition}{Definition}

\newcommand{\D}{\mathbb{D}}
\newcommand{\8}{\infty}
\newcommand{\h}{H(\D)}
\newcommand{\V}{\mathcal{V}_n}

\newcommand{\hh}{\mathcal{H}^2}

\newcommand{\z}{\prescript{}{0}{\mathcal{S}^2_{n}}}
\newcommand{\sh}{\mathcal{S}_n^2}

\date{}

\begin{document}


\title{The lattices of invariant subspaces of a class of operators on the Hardy space}

\author{\v{Z}eljko \v{C}u\v{c}kovi\'c and Bhupendra Paudyal}

\maketitle

\begin{abstract}

 In the authors' first paper, Beurling-Rudin-Korenbljum type characterization of the closed ideals in a certain algebra of holomorphic functions was used to describe the lattice of invariant subspaces of the shift plus a complex Volterra operator. Current work is an extension of the previous work and it describes the lattice of invariant subspaces of the shift plus a positive integer multiple of the complex Volterra operator on the Hardy space. Our work was motivated by a paper by Ong who studied the real version of the same operator.
\end{abstract}

\maketitle
\section{Introduction}
Characterizing all closed invariant subspaces of a bounded operator acting on a separable Hilbert space has typically been a challenging problem. Despite the researchers continuous efforts for the last several decades in this area, only a few operators have had their lattice of closed invariant subspaces completely described. One of the most influential results in this direction is Beurling's Theorem \cite{Be48}. It describes the lattice of the shift operator acting on the Hardy space.  Sarason \cite{Sa65} characterized all closed invariant subspaces of the real Volterra operator on $L^2[0,1]$ utilizing the Beurling's result. He also studied the invariant subspaces of multiplication plus real Volterra operator on $L^p[0,1]$ in \cite{Sa74}. Following Sarason's technique Ong \cite{ON81} studied invariant subspaces of the shift plus a multiple of the Volterra operators on $L^2[0,1]$. Aleman and Korenblum \cite{Al08} characterized the invariant subspaces of the complex Volterra operator acting on the Hardy space. In our first paper, \v{C}u\v{c}kovi\'c and Paudyal \cite{BP15}, we used Beurling-Rudin-Korenbljum type characterization of the closed ideals in certain algebra of holomorphic functions to describe the lattice of invariant subsapces of the shift plus a complex Volterra operator. Motivated by Ong's work \cite{ON81}, in this paper we extend our previous result and obtain a complex analogue of Ong's result.
 
Suppose $\D$ is the unit disk in the complex plane and $\h$ be the space of holomorphic functions on the unit disk. Let us denote the space of bounded holomorphic functions by $H^\8$ and the Hardy space by $\hh$. For $f(z)=\sum_{n=0}^\infty a_n z^n \in \hh$, we define a norm
\begin{equation} \label{eq1}
\|f\|^2_{\hh}= \sum_{n=0}^\8|a_n|^2.
\end{equation}

  The main result of this work characterizes the lattice of closed invariant subspaces of one parameter family of the shift plus an integer times the Volterra operator, defined by 
  
  \begin{equation}\label{eq2}
(T_n f)(z)=z f(z)+n\int_0^z f(w) dw, \hspace{.5cm} ~ z\in \D, n\in \mathbb{N}.
\end{equation}
 Since both the shift operator and complex Volterra operator acting on $\hh$ are bounded on $\hh$, so are $\{T_n\}_n$.

The idea of our paper is to convert the invariant subspace problem into an equivalent problem of characterizing the closed ideals of certain Banach algebras. Then we use the result from Korenbljum \cite{KO72} and \cite{KO72A} that characterized these ideals. Korenbljum's work in turn relates to the work of Rudin \cite{RU57} and his main theorem on the structure of the closed ideals of the disk algebra. These techniques play an important role in our work.  Montes-Rodriguez, Ponce-Escudero and Shkarin \cite{MO10} used a similar idea to characterize invariant subspaces of certain composition operators acting on the Hardy space. 

\section{Banach Algebra $\sh$} 

 Let us define a space
\begin{equation}\label{eq3}
 \sh =\{f\in \hh: f^{(n)}\in \hh\}, ~\text{where}~n\in\mathbb{N}~\text{and}~ f^{(n)}(z)=\frac{d^n}{dz^n}f(z),
\end{equation}
with the norm given by  \begin{equation} \label{eq4}
 \|f\|^2_{\sh} = \|f^{(n)}\|^2_{\hh }+\|f\|^2_{\hh}.
\end{equation}

 As in our previous paper \cite{BP15}, the space $\sh$ becomes a Banach space under this norm and polynomials are dense in $\sh$. It is clear from the definition that $\mathcal{S}_{n+1}^2\subset \mathcal{S}_n^2$. We now show that each $\mathcal{S}_n^2$ is contained in the space of bounded holomorphic functions $ H^\8$. For this, it is enough to show that $\mathcal{S}_1^2 \subset H^{\infty}$. Let us recall that \[|f(z)|\leq \frac{\|f\|_{\hh}}{\sqrt{1-|z|^2}},~ f\in \hh, z\in\D .\] 
 
 Since $f'\in \hh$ for $ f\in \mathcal{S}_1^2$, we have
 \begin{align}\label{eq5}
|f(z)-f(0)|
 =&~|\int_0^1 zf'(tz) dt|\nonumber \\
 \leq &~\|f'\|_{\hh} \int_0^1 \frac{|z|}{\sqrt{1-|tz|^2}} dt \nonumber \\
  \leq &~\|f'\|_{\hh} \int_0^1 \frac{|z|}{\sqrt{1-t|z|}} dt \nonumber\nonumber\\
\leq & ~2 \|f'\|_{\hh}.
\end{align}

For any $f\in  \mathcal{S}_1^2$, this further shows that 
\begin{align}\label{eq6}
\|f\|_\8\leq &~ 2\|f'\|_{\hh}+|f(0)|\nonumber\\
\leq &~ 2\|f'\|_{\hh}+ 2\|f\|_{\hh}=2\|f\|_{\mathcal{S}_1^2}.
\end{align}
Now it follows that $\mathcal{S}_{n+1}^2\subset \mathcal{S}_n^2 \subset H^\infty$. In fact we will show that there exists $C_k > 1$ (depending on $k$) such that  $ \|f\|_{\mathcal{S}_{k}^2} \leq C_k\|f\|_{\mathcal{S}_{k+1}^2}$for $f\in \mathcal{S}_{k+1}^2 $. 

\begin{align}\label{eq7}
 \|f\|_{\mathcal{S}_k^2}^2 =&~ \|f^{(k)}\|_{\hh}^2+\|f\|_{\hh}^2\nonumber\\ 
 = &~\sum_{n=k}^\infty |n (n-1)(n-2)\dots (n-k+1) a_n|^2+\|f\|_{\hh}^2\nonumber\\
  =&~ \big(k!\big)^2 |{a_k}|^2+\sum_{n=k+1}^\infty |n (n-1)(n-2)\dots (n-k+1) a_n|^2+\|f\|_{\hh}^2\nonumber\\
   \leq&~ \big(k!\big)^2\|f\|_{\hh}^2+\|f^{(k+1)}\|_{\hh}^2+\|f\|_{\hh}^2~~~(\text{using the fact $|a_k|^2\leq\|f\|_{\hh}^2$ })\nonumber\\
   \leq&~ C^2_k \|f\|_{\mathcal{S}_{k+1}^2}^2~~~(\text{for some constant $C_k > 1$}).
  \end{align}
 Furthermore, the pointwise multiplication makes $\sh$ a Banach algebra. We prove it by induction. If $n=1$ and $f, g$ are in $\mathcal{S}_1^2$, then
\begin{align}\label{eq8}
\|fg\|^2_{\mathcal{S}_1^2} =& \|(fg)'\|^2_{\hh }+\|fg\|^2_{\hh}\nonumber\\
=&\|gf'+fg'\|^2_{\hh}+\|fg\|^2_{\hh}\nonumber\\
\leq&\|gf'\|^2_{\hh}+2 \|gf'\|_{\hh}\|fg'\|_{\hh}+\|fg'\|^2_{\hh}+ \|fg\|^2_{\hh}\nonumber\\
\leq&\|g\|^2_\8 \|f'\|^2_{\hh}+2 \|f\|_\8\|g\|_\8\|f'\|_{\hh}\|g'\|_{\hh} \nonumber\\ 
&+\|f\|^2_\8\|g'\|^2_{\hh} + \|g\|^2_\8\|f\|^2_{\8}\nonumber \\
\leq & 4\|g\|^2_{\mathcal{S}_1^2} \|f\|^2_{\mathcal{S}_1^2}+ 8\|g\|^2_{\mathcal{S}_1^2}\|f\|^2_{\mathcal{S}_1^2} + 4\|g\|^2_{\mathcal{S}_1^2} \|f\|^2_{\mathcal{S}_1^2}~~\text{(using\eqref{eq6})}\nonumber\\
\leq & 16 \|f\|^2_{\mathcal{S}_1^2}\|g\|^2_{\mathcal{S}_1^2}.
\end{align}
We assume for $n=k$ there is a $C>0$ such that $\|fg\|_{\mathcal{S}_{k}^2}\leq C \|f\|_{\mathcal{S}_{k}^2}\|g\|_{\mathcal{S}_{k}^2}$ for all $f,g \in\mathcal{S}_{k}$. Now we want to prove that an analogues inequality is true for $n=k+1$. Let us assume $f, g\in\mathcal{S}_{k+1}\subset \mathcal{S}_{k}$ which implies $f', g'\in \mathcal{S}_{k}$ and we have
\begin{align}\label{eq9}
\|fg\|^2_{\mathcal{S}_{k+1}^2} =&~ \|(fg)^{(k+1)}\|^2_{\hh }+\|fg\|^2_{\hh}\nonumber\\
=&~ \|((fg)')^{(k)}\|^2_{\hh}+\|fg\|^2_{\hh}\nonumber\\
\leq&~ \|(fg)'\|^2_{\mathcal{S}_k^2}+\|fg\|^2_{\hh}\nonumber\\
=&\|gf{'}+fg{'}\|^2_{\mathcal{S}_k^2}+\|fg\|^2_{\hh}\nonumber\\
\leq&\|gf'\|^2_{\mathcal{S}_k^2}+2 \|gf'\|_{\mathcal{S}_k^2}\|fg'\|_{\mathcal{S}_k^2}+\|fg'\|^2_{\mathcal{S}_k^2}+ \|fg\|^2_{\mathcal{S}_k^2}\nonumber\\
\leq& C^2\|g\|^2_{\mathcal{S}_k^2}\|f'\|^2_{\mathcal{S}_k^2}+2C^2 \|g\|_{\mathcal{S}_k^2}\|f'\|_{\mathcal{S}_k^2}\|f\|_{\mathcal{S}_k^2}\|g'\|_{\mathcal{S}_k^2}+C^2\|f\|^2_{\mathcal{S}_k^2}\|g'\|^2_{\mathcal{S}_k^2}+ C^2\|f\|^2_{\mathcal{S}_k^2}\|g\|^2_{\mathcal{S}_k^2}
\end{align}
Further,
\begin{align}\label{eq9'}
 \|f'\|_{\mathcal{S}_k^2}^2 =&~ \|f^{(k+1)}\|_{\hh}^2+\|f'\|_{\hh}^2\nonumber\\ 
\leq &\|f\|_{\mathcal{S}_{k+1}^2}^2+\|f\|_{\mathcal{S}_{1}^2}^2\nonumber\\
\leq &\|f\|_{\mathcal{S}_{k+1}^2}^2+D_k\|f\|_{\mathcal{S}_{k+1}^2}^2~~~(\text{for some constant  $D_k>1$, from \eqref{eq7}})\nonumber\\
=& (D_k+1)\|f\|_{\mathcal{S}_{k+1}^2}^2.
  \end{align}
Hence applying Equations \eqref{eq7} and \eqref{eq9} on  \eqref{eq9'}, we see that $\|fg\|_{\mathcal{S}_{k+1}^2}\leq M_k\|f\|_{\mathcal{S}_{k+1}^2}\|g\|_{\mathcal{S}_{k+1}^2}$ for some $M_k>0$. Thus this discussion suggests that the multiplication is continuous in each factor separately. Therefore there exists a norm equivalent to $\|.\|_{\sh}$, for which $\sh$ is a Banach algebra (see \cite[page 212]{Ka04}). Keeping this fact in mind, we do not differentiate between the norm $\|.\|_{\sh}$ and its equivalent norm that makes $\sh$ a Banach algebra.

Now let us define a subalgebra of $\sh$ given by 
\begin{equation}\label{eq10}
 \z=\{f\in\sh : f^{(i)}(0)=0, ~0\leq i\leq n-1\}
\end{equation} 
Assume $g\in\overline{\z}$ under the norm defined on \eqref{eq4}. This implies there exists a sequence $g_l\in\z$, $l\in\mathbb{N}$ such that $g_l$ converges to $g$ in $\sh$ norm. This implies $g_l$ converges to $g$ in $\hh$ norm. Since $g^{(i)}_l(0)=0$ for all $l$, it follows that $g^{(i)}(0)=0, ~ \text{for}~ 0\leq i\leq n-1$. Hence $\z$ is a closed subalgebra which in fact is a Banach algebra. Furthermore from the definition, we can see that $z^i\notin \z$ for $0\leq i\leq n-1$. Clearly, 
\begin{equation}\label{eq11}
\sh=[1]\bigoplus [z] \bigoplus [z^2]\bigoplus \dots \bigoplus [z^{n-2}]\bigoplus [z^{n-1}] \bigoplus \z,
\end{equation}
 and hence \[\z= \overline{span}\{z^{n+k-1}:k\in \mathbb{N}\}.\]

\section{On the Similarity}

For a positive integer $n$, let us consider the complex Riemann-Liouville integral $\V$, see \cite[p.421]{Sa93},

\begin{equation}\label{eq12} (\V f)(z) = \frac{1}{\Gamma(n)}\int_0^z (z-w)^{n-1} f(w)\, dw,~ f\in \hh,~z\in \D.
\end{equation}

The complex Riemann-Liouville integral  in fact can be given by an iterated integral as shown below. 
\begin{align*}
(\V f)(z)=\frac{1}{\Gamma(n)}\int_0^z (z-w)^{n-1} f(w) \,dw
= \int_0^z \int_0^{\sigma_1} \cdots \int_0^{\sigma_{n-1}} f(\sigma_{n}) \, \mathrm{d}\sigma_{n} \cdots \, \mathrm{d}\sigma_2 \, \mathrm{d}\sigma_1.
\end{align*}
A proof of this is given by induction. For $n=1$, it is obvious. Now, suppose this is true for $n$, that is:
\[(\mathcal{V}_n f)(z)=\dfrac{1}{\Gamma(n)}\int_0^z (z-w)^{n-1}f(w) \,dw\\
=\int_0^z \int_0^{\sigma_1} \cdots \int_0^{\sigma_{n-1}} f(\sigma_{n}) \, \mathrm{d}\sigma_{n} \cdots \, \mathrm{d}\sigma_2 \, \mathrm{d}\sigma_1.\]

 Applying the induction hypothesis,
\begin{align*}
(\mathcal{V}_{n+1} f)(z)=&\int_0^z \int_0^{\sigma_1} \cdots \int_0^{\sigma_{n-1}}\int_0^{\sigma_{n}} f(\sigma_{n+1}) \, \mathrm{d}\sigma_{n+1}\,\mathrm{d}\sigma_{n}  \cdots \, \mathrm{d}\sigma_2 \, \mathrm{d}\sigma_1 \\
& \left( \text{Assume}  ~g(\sigma_n)=\int_0^{\sigma_n} f(\sigma_{n+1})d\sigma_{n+1}\right)\\
=&\int_0^z \int_0^{\sigma_1} \cdots \int_0^{\sigma_{n-1}} g(\sigma_{n}) \, \mathrm{d}\sigma_{n}\,\mathrm{d}\sigma_{n-1}  \cdots \, \mathrm{d}\sigma_2 \, \mathrm{d}\sigma_1\\
=&\;\dfrac{1}{\Gamma(n)}\int_0^z (z-w)^{n-1}g(w) \,dw\\
=&\;\dfrac{z^n}{\Gamma(n)}\int_0^1 (1-t)^{n-1}g(tz) \,dt,~~\text{assuming} ~w=tz\\
=&\;\dfrac{z^n}{\Gamma(n)} \left[ -\dfrac{(1-t)^n}{n}\, g(tz)\Biggr|_0^1+\int_0^1\dfrac{(1-t)^{n}}{n}g'(tz)\;z\;dt\right]\\
=&\;\dfrac{1}{\Gamma(n)} \int_0^1\dfrac{(z-tz)^{n}}{n}\;f(tz)\; z\;dt\\
=&\;\dfrac{1}{\Gamma(n+1)} \int_0^z(z-w)^{n}f(w) \;dt,~~(\text{pluging back}~ tz=w)
\end{align*}

Now in the next theorem, we will show that the complex Riemann-Liouville operator is bounded isomorphism from $\hh$ onto $\z$. It will be seen that the collection of all the closed invariant subspaces of $T_n$ is in one-to-one correspondence with the collection of all closed ideals of $\z$. In fact these closed ideals are precisely the collection of invariant subspaces of multiplication operator $M_z$ acting on $\z$. 

\begin{theorem}\label{thm1}
Let $\V$ be the complex Riemann-Liouville operator on $\hh$. Then the following statements hold:
\renewcommand\theenumi{\roman{enumi}}
\begin{enumerate}
\item[(i)]  Range of $\V= \z$.
\item[(ii)] $\V$ is a bounded isomorphism from $\hh$ onto $\z$, and its inverse is $\dfrac{d^n}{dz^n}$.
\item[(iii)] The operator $T_n$ acting on $\hh$ is similar under $\V$ to the multiplication operator $M_z$ acting on $\z$. 
\end{enumerate}
\end{theorem}
\begin{proof}
\renewcommand\theenumi{\roman{enumi}}
\begin{enumerate}
\item[(i)] 

Let $g$ be in the range of $\V$, then there exists $ f\in\hh$ such that
\begin{align*}
\displaystyle{g(z)= (\V f)(z)}=&~\frac{1}{\Gamma(n)}\int_0^z (z-w)^{n-1} f(w) dw\\
= &\int_0^z \int_0^{\sigma_1} \cdots \int_0^{\sigma_{n-1}} f(\sigma_{n}) \, \mathrm{d}\sigma_{n} \cdots \, \mathrm{d}\sigma_2 \, \mathrm{d}\sigma_1.
\end{align*}
The last expression is obtained from Cauchy formula for $n$-repeated integral. This repeated integral can be differentiated $n$-times. Now it is easy to see that $g^{(n)}=f\in \hh$ and $g^{(i)}(0)=0$ for $ 0\leq i\leq n-1$. Therefore $g\in \z$.
 
Conversely, suppose that $g \in \z$. Then $g^{(n)}\in \hh$ and $g^{(i)}(0)=0$ for $ 0\leq i\leq n-1$. Now
\begin{align*}
g(z)= &\int_0^z \int_0^{\sigma_1} \cdots \int_0^{\sigma_{n-1}} g^{(n)}(\sigma_{n}) \, \mathrm{d}\sigma_{n} \cdots \, \mathrm{d}\sigma_2 \, \mathrm{d} \sigma_1\\
 =&\frac{1}{\Gamma(n)}\int_0^z (z-w)^{n-1} g^{(n)}(w) dw =(\V g ^{(n)})(z).
\end{align*}
Therefore $g$ belongs to Range of $\V$, which completes the proof.
\item[(ii)] The complex integral operator is linear. To show boundedness, let $f\in\hh$. 
 \begin{align*}
\|\V (f)\|_{\sh}=&~\|\V (f)\|_{\hh}+\|(\V (f))^{(n)}\|_{\hh}\\
\leq &~\|f\|_{\hh} +\|f\|_{\hh}\\=&~2\|f\|_{\hh}.
\end{align*} 
To show $\V$ is one-one, assume that $f_1$ and  $f_2$ belong the Hardy space, and also assume that  \[ \int_0^z (z-w)^{n-1}f_1(w)dw=\int_0^z (z-w)^{n-1}f_2(w)dw, \hspace{1cm} \text{for~all}~ z \in\D.\] Differentiating both sides we see that $f_1=f_2$ and hence $\V$ is one-one.\\

 Since we already have seen that $\V (f^{(n)})=f=(\V f)^{(n)}$, $\V$ is a bounded isomorphism with inverse $\V^{-1}=\dfrac{d^n}{dz^n}$.
\item[(iii)]
Take $f\in \hh$, then there is $g\in \z$ such that $\V f =g$. We get $f(z)=(\V ^{-1}g)(z)= g^{(n)}(z)$ and therefore $\displaystyle g^{(n-1)}(z)=\int_0^z f(w) dw.$
\begin{align*}
\dfrac{d^n}{dz^n}(M_zg(z))=\dfrac{d^n}{d^n}(zg(z))=&~\dfrac{d^{n-1}}{d^{n-1}}\Big(g(z)+z g'(z)\Big)\\
=&~g^{(n-1)}(z)+\dfrac{d^{n-2}}{d^{n-2}}\Big(g'(z)+z g''(z)\Big)\\
=&~g^{(n-1)}(z)+g^{(n-1)}(z)+\dfrac{d^{n-3}}{d^{n-3}}\left(g''(z)+z g'''(z)\right)\\
=&~\dots\\
=&~ng^{(n-1)}(z)+ z g^{(n)} (z)\\
=&~ z f(z)+n \int_0^z f(w) dw
=(T_nf)(z).
\end{align*}
Now applying $\V$ on the both side, we see that
\begin{align*}
M_z g(z)=&(\V T_n)f(z)\\M_z \V f(z)=&(\V T_n)f(z).
\end{align*}\end{enumerate}
So, $\V T_n=M_z\V$ and $\V T_n\V^{-1}=M_z$. 
 That is to say $\V$ transforms the operator $T_n$ into the multiplication operator $M_z$ on $\z$.
 \end{proof}

  We can summarize the theorem by the following commutative diagram
\begin{center}
\begin{tikzpicture}
\coordinate[label=above :$\z$](A) at (6,-.2);
\coordinate[label=below:$\z$] (B) at (6,4.2);
\coordinate[label= below:$\hh$] (C) at (0,4.2);
\coordinate[label=above:$\hh$] (D) at (0,-.2);
\draw[thick,->] (6,3.2)--(6,0.8);
\draw[thick,->] (0.8,4)--(5.2,4);
\draw[thick,->] (0,3.2)--(0,0.8);
\draw[thick,->] (0.8,0)--(5.2,0);
\draw[thick,->] (.8,3.2)--(5.2,.8);
\draw (6,0.8) -- (6,3.2) node[right,midway]{$M_z$ };
\draw (5.2,4) -- (0.8,4) node[above,midway]{$\V$ };
\draw (0.8,0)--(5.2,0) node[below,midway] {$ \V$ };
\draw (0,3.2)--(0,0.8) node[left,midway] {$T_n$ };
\draw (.8,3.2)--(5.2,.8) node[above,midway] {$\V T=M_z\V$ };
\end{tikzpicture}
\end{center}

\section{Invariant subspace}

We begin this section by proving the following proposition which  provides a relationship between the closed invariant subspaces of a multiplication operator $M_z$ and the closed ideals of $\sh$.

 \begin{proposition}
The closed invariant subspaces of $M_z$ on $\sh$ are exactly the closed ideals of $\sh$.
\end{proposition}

\begin{proof}

Suppose $L$ is an invariant subspace of $M_z$ in $\sh$. Then since $zL\subset L$ and $1L = L$, we have $p_n L \subset L$ for any polynomial $p_n$. But the polynomials are dense in $\sh$, hence for any $h$ in $\sh$ we have that 
$h = \lim p_n$ for some sequence $p_n$.  Since $L$ is a closed subspace of $\sh$, for any $l$ in $L$ we have $hl = \lim p_n l$ is in $L$, and therefore it follows that $L$ is a closed ideal. On the other hand, each closed ideal $L$ is invariant with respect to $M_z$. This completes the proof

\end{proof}

Now we present some definitions and a theorem, which are at the heart of our Main Theorem. Detailed discussion can be found in Korenbljum work \cite{KO72A}. 
  
 \begin{definition}\label{def2}
 Let $K$ be a closed subset of the unit circle $\partial\D$. For any inner function $G$ we say $G$ is \textit{associated with} $K$ if 
\begin{itemize}
\item[(a)] if $a_1,a_2,...$ are the zeros of $G(z)$ in the open disk, then all the limit points of $\{a_k\}$ belong to $K$;
\item[(b)] the measure determining the singular part of $G$ is supported on $K$.
\end{itemize}
 \end{definition}
 
\begin{definition}\label{def3}
Let $n\geq 1$ be fixed, and $K_0, K_1,\dots, K_{n-1}$ be closed sets of the unit circle $\partial \D$ that satisfy the conditions
\begin{enumerate}
\item[(a)] $K_0 \supset K_1\supset\dots \supset K_{n-1}$;
\item[(b)] $K_0\setminus K_{n-1}$ is isolated;
\item[(c)] $G$ is an inner function associated with $K_{n-1}$ as in Definition \ref{def2}.
\end{enumerate}
Then we denote by $I\{G;K_0, K_1,\dots, K_{n-1}\}$ the closed ideal of the algebra $\sh$ consisting of all functions $ f$ that satisfy the conditions 
\begin{enumerate}
\item[(i)] $ f(z)=f'(z)=f^{(2)}(z)=\dots= f^{(i)}(z)=0, z\in K_i; 0\leq i\leq n-1;$
\item[(ii)] $G$ divides the inner part of $f$.
\end{enumerate}
\end{definition} 
  \begin{remark}
If one of the $K_i$ in the previous definition has Lebesgue arc length measure non-zero, then the corresponding ideal $I\{G;K_0, K_1,\dots, K_{n-1}\}$ is the zero ideal.
\end{remark}
 \begin{remark} The ideal $K_0, K_1,\dots, K_{n-1}$ is distinct from zero if and only if the condition 
 \begin{equation}\label{eq13}
 \int_{\partial \D} \log\rho(z)|dz| >-\infty
 \end{equation}
 holds, where $\rho(z)=\min_{\zeta \in K}|z-\zeta|,~K=K_0 \cup \{a_k\}.$
 \end{remark}

 \begin{theorem}\cite[Main Theorem]{KO72A}\label{thm2} Let $I$ be an arbitrary nonzero ideal of the algebra $\sh (n\geq 1$ fixed). Suppose moreover that 
 \begin{enumerate}
 \item[(a)] $K_i=\cap_{f\in I} K_i(f), 0\leq i\leq n-1$ where $K_i(f)=\{ z:z\in \partial\D, f(z)=f'(z)=f^{(2)}(z)=\dots= f^{(i)}(z)=0\}$;
 \item[(b)] $G(z)$ is the greatest common divisor of the inner parts of the functions $f\in I$.
 \end{enumerate}
 Then $I=I\{G;K_0, K_1,\dots, K_{n-1}\}$, and the sets $K_i$ and the function $G$ satisfy conditions $(a)-(c)$ of Definition \ref{def2}, and also satisfy Equation \eqref{eq13}.
 
 \end{theorem}

 The theorem above from Korenbljum \cite{KO72A} characterizes all closed ideals of $\sh$. In particular it says that all these ideals are of the form of $I\{G;K_0, K_1,\dots, K_{n-1}\}$ as defined in Definition \ref{def3}. Furthermore from our discussion above, these ideals are the closed invariant subspaces of $M_z$ on $\sh$.  Now we present our main theorem.

\begin{theorem} [{\bf Main Theorem}]
\label{mthm}
For some $n\in\mathbb{N}$, let $T_n$ be an operator 
\[ (T_nf)(z) =zf(z)+ n \int_0^z f(w) dw.\] defined on $\hh$. Then the lattice of closed invariant subspaces
\[Lat ~T_n=\Big\{S \subset \hh :S=\{f^{(n)}:f \in I\{G;K_0, K_1,\dots, K_{n-1}\}\cap \z\}~\text{for}~G, K_i \text {~as in Definition \ref{def3}} \Big\}.\]

\end{theorem}

\begin{proof}

Suppose $L$ is an invariant subspace of $M_z|_{\z}$. Then it is also an invariant subspace of $M_z$ on $\sh$. By Proposition 1, $L$ is a closed ideal of $\sh$ and hence, by Korenbljum, $L=I\{G;K_0, K_1,\dots, K_{n-1}\}$ for some $G$ and $K_i$. But since $L$ is also in $\z$, so $L=I\{G;K_0, K_1,\dots, K_{n-1}\}\cap \z$. Conversely,  $I\{G;K_0, K_1,\dots, K_{n-1}\}$ is an ideal of $\sh$, so by Proposition 1 it is a $z$ invariant subspace of $\sh$ and $\z$ is also $z$ invariant, hence their intersection $L=I\{G;K_0, K_1,\dots, K_{n-1}\}\cap \z$ is an invariant subspace of $M_z|_{\z}$. It follows that every invariant subspace of $M_z|_{\z}$ is of the form $L=I\{G;K_0, K_1,\dots, K_{n-1}\}\cap \z$ for some $G, K_i$. Hence
 \[Lat~ M_z|_{\z}=\Big\{I\{G;K_0, K_1,\dots, K_{n-1}\} \cap \z: G, K_i~ \text{ as in Definition \ref{def3}}\Big\}.\]
Since $ \V T_n=M_z\V$, the closed invariant subspaces of $T_n$ on $\hh$ are of the form $\V^{-1} (I\{G;K_0, K_1,\dots, K_{n-1}\} \cap\z)$. Recall $\V^{-1}=\dfrac{d^n}{dz^n}$ and hence
\[Lat ~T_n=\Big\{S \subset \hh :S=\{f^{(n)}:f \in I\{G;K_0, K_1,\dots, K_{n-1}\}\cap \z\}~\text{for}~G, K_i~\text{~as~in~Definition \ref{def3}} \Big\}.\]

\end{proof}

Department of Mathematics and Statistics, The University of Toledo, Toledo, Ohio\\
{\it E-mail address: zeljko.cuckovic@utoledo.edu}\\

Mathematics and Computer Science Department, Central State University, Wilberforce, Ohio\\
{\it Email address: bpaudyal@centralstate.edu}

\end{document}